\documentclass[11pt,reqno]{amsart}
\usepackage[T1]{fontenc} 
\usepackage{cite}
\usepackage{amssymb,amsmath}
\usepackage[english]{babel}
\usepackage[arrow,curve,matrix,tips]{xy}
\usepackage{enumerate}
\usepackage{dsfont}
\usepackage{esint}
\usepackage{ifthen}
\usepackage{nicefrac}
\usepackage{color}
\usepackage{enumitem,kantlipsum}
\usepackage{todonotes}
\usepackage[colorlinks=true,linkcolor=blue]{hyperref}

\setlist[enumerate]{leftmargin=*, label=(\roman*)}
\numberwithin{equation}{section}

\def\E{\mathbb{E}}
\def\F{\mathcal{F}}
\def\N{\mathbb{N}}
\def\n{\mathcal{N}}
\def\P{\mathbb{P}}
\def\S{\mathbb{S}}
\def\R{\mathbb{R}}
\def\T{\mathcal{T}}
\renewcommand{\d}{{\rm d}}
\def\L{\mathcal{L}}

\def\epsilon{\varepsilon}
\def\phi{\varphi}
\def\one{\mathds{1}}

\DeclareMathOperator{\id}{id}

\DeclareMathOperator{\loc}{loc}
\DeclareMathOperator{\supp}{supp}
\DeclareMathOperator{\sym}{sym}

\newtheorem{theorem}{Theorem}[section]
\newtheorem{lemma}[theorem]{Lemma}

\newtheorem{corollary}[theorem]{Corollary}
\theoremstyle{definition}
\newtheorem{definition}[theorem]{Definition}
\newtheorem{example}[theorem]{Example}

\newtheorem{remark}[theorem]{Remark}
\newtheorem{assumption}[theorem]{Assumption}

\begin{document}

\title[Viscous  Hamilton-Jacobi equations in exponential Orlicz hearts]
{Viscous Hamilton-Jacobi equations in \\ exponential Orlicz hearts} 

\author{Jonas Blessing}
\address{Department of Mathematics and Statistics, University of Konstanz}
\email{jonas.blessing@uni-konstanz.de}

\author{Michael Kupper}
\address{Department of Mathematics and Statistics, University of Konstanz}
\email{kupper@uni-konstanz.de}

\date{\today}

\thanks{We thank Robert Denk and Max Nendel  for helpful discussions.}

\begin{abstract}
We provide a semigroup approach to the viscous Hamilton-Jacobi equation. It turns out that 
exponential Orlicz hearts are suitable spaces to handle the (quadratic) non-linearity of the 
Hamiltonian. Based on an abstract extension result for nonlinear semigroups on spaces of continuous 
functions, we represent the solution of the viscous Hamilton-Jacobi equation as a strongly 
continuous convex semigroup on an exponential Orlicz heart. As a result, the solution depends 
continuously on the initial data.  We further determine the symmetric Lipschitz set which is 
invariant under the semigroup. This  automatically yields a priori estimates and regularity in 
Sobolev spaces.  In particular, on the domain restricted to the symmetric Lipschitz set, 
the generator can be explicitly determined and linked with the viscous Hamilton-Jacobi equation.

 
\end{abstract}

\maketitle

\setcounter{tocdepth}{1}

\section{Introduction}

In this article, we provide a semigroup approach to the viscous Hamilton-Jacobi equation
\begin{equation} \label{eq:CP}
	\begin{cases}
		\partial_t u(t,x)=\frac{1}{2}\Delta u(t,x)+H(\nabla u(t,x)),  & (t,x)\in (0,\infty)\times\R^d, \\
		u(0,x)=f(x),  & x\in\R^d,
	\end{cases}
\end{equation}
where $H\colon\R^d\to\R$ is a convex function growing at most quadratically.  A semigroup is a 
family $(S(t))_{t\geq 0}$ of (not necessarily linear) operators $S(t)\colon X\to X$ such that 
$S(0)=\id_X$ and $S(s+t)=S(s)S(t)$ for all $s,t\geq 0$.  In case that $X$ is a Banach lattice,  the 
semigroup is called convex if $S(t)(\lambda x+(1-\lambda)y)\leq \lambda S(t)x+(1-\lambda)S(t)y$
for all $x,y\in X$, $t\ge 0$ and $\lambda\in [0,1]$.  Convex semigroups have been systematically 
studied in~\cite{DKN19}.  There,  it is shown that convex semigroups on $L^p$-like spaces fulfil
properties which are similar to the well-established theory of linear semigroups.  In particular,  the 
generator of a convex semigroup is a closed operator, its domain is invariant under the semigroup,  
and the associated abstract Cauchy problem is classically well-posed.  Since the differential operator 
on the right-hand side of equation~\eqref{eq:CP} is convex, the aim of this paper is to construct an 
associated convex semigroup for which we can apply the results of~\cite{DKN19}.  In the special case 
where $H$ is sublinear, the existence of such a semigroup on $L^p$, for $p\in [1,\infty)$,  has been 
established in~\cite[Example~5.3]{DKN19}.  However,  as soon as $H$ growths superlinear,  the 
approach in~\cite{DKN19} fails in $L^p$.  In this case,  a suitable space is an exponential Orlicz heart.  
This choice is motivated by the fact,  that in the special case $H(x):=|x|^2/2$,  an explicit solution of 
equation~\eqref{eq:CP} is given by the formula
\begin{equation} \label{eq:square}
	u(t,x):=\log\left((2\pi t)^{-\frac{d}{2}}\int_{\R^d}\exp(f(x+y))e^{-\frac{|y|^2}{2t}}\,\d y\right).  
\end{equation}
Hence,  it seems natural to consider functions that are exponentially integrable.  In fact,  for $H(x)=a|x|^p$ 
with $a>0$,  $p\geq 2$ and $f\in L^1_{\loc}$,  it was shown in~\cite[Proposition~3.1]{BSW2002} that 
if equation~\eqref{eq:CP} has a classical solution $u\in C^{1,2}((0,\infty)\times\R^d)$ with 
$\lim_{t\downarrow 0}u(t)=f$ in $L^1_{\loc}$, then $\exp(af)\in L^1_{\loc}$.  Similar integrability
assumptions on the initial data have been made in~\cite{FPR2001,DGLS2006}.  

The construction of the semigroup consists of two steps.  First, our approach in~\cite{BK20} 
yields a semigroup $(S(t))_{t\geq 0}$ on the space $C_0$ of all continuous functions which vanish 
at infinity. The idea is not to linearise the equation,  but to find a family $(I(t))_{t\geq 0}$ of convex 
operators $I(t)\colon C_0\to C_0$ which satisfy
\[ \lim_{t\downarrow 0}\left\|\frac{I(t)f-f}{t}-\frac{\Delta f}{2}-H(\nabla f)\right\|_\infty=0 \]
for $f$ sufficiently smooth.  We make the ansatz
\begin{equation} \label{eq:def I}
	\big(I(t)f\big)(x):=\sup_{\lambda\in\R^d}\left((2\pi t)^{-\frac{d}{2}}\int_{\R^d} 
	f(x+y+\lambda t)e^{-\frac{|y|^2}{2t}}\,\d y-L(\lambda)t\right),  
\end{equation}
where $L$ denotes the convex conjugate of $H$.  Then,  we iterate these operators along the dyadic
numbers and obtain the semigroup as the limit along a subsequence, i.e., 
\begin{equation} \label{eq:chernoff}
	\lim_{l\to\infty}\|S(t)f-I(2^{-n_l}t)^{2^{n_l}}f\|_\infty=0. 
\end{equation}
By construction,  the key properties of $(I(t))_{t\geq 0}$ transfer to $(S(t))_{t\geq 0}$.  For the purpose 
of this paper,  the results gathered in Theorem~\ref{thm:gexp C_0} are sufficient,  see also~\cite{BK20} 
for further details.  Formulas of type~\eqref{eq:chernoff} are called Chernoff approximation
or Trotter formula,  see~\cite{chernoff1968,chernoff1974,trotter1958,trotter1959}.  An alternative approach 
is the so-called Nisio construction,  where the semigroup $(S(t))_{t\geq 0}$ is obtained as a monotone 
limit,  see~\cite{Nisio76,DKN20a,NR19}.  We also want to mention the classical approach to nonlinear 
semigroups based on maximal monotone or m-accretive operators, 
see~\cite{Barbu2010,BC1991,Brezis71,Evans78,Kato67}.  However, the required assumptions are in general 
difficult to verify for convex Hamilton-Jacobi-Bellman-type equations,  see~\cite[Example~4.2]{DKN20} 
for a counterexample.  While spaces of continuous functions are useful for the construction of convex 
semigroups,  they have the disadvantage that the domain of the generator is in general not invariant,
see~\cite[Example~4.4]{DKN20} for a counterexample.  However, the invariance of the domain (or a 
suitable subspace thereof) is essential to link the semigroup with the associated Cauchy problem. This 
can be achieved by extending the semigroup to a larger $L^p$-like space,  which is the main focus of this 
article.  

Second,  we extend $(S(t))_{t\geq 0}$ from $C_0$ to the Orlicz heart $M^\Phi$ of all functions 
that are exponentially integrable.  The key is to find a dominating family of operators 
$T(t)\colon M^\Phi\to M^\Phi$, which satisfy $|I(t)f|\leq T(t)|f|$ for all $t\geq 0$ and $f\in C_0\cap M^\Phi$.  
In our case,  the operators $T(t)$ will be defined by a formula similar to~\eqref{eq:square},  where the 
exponential will be replaced by a suitable function~$\phi$.  From the construction of $(S(t))_{t\geq 0}$ 
outlined above,  we obtain $|S(t)f|\leq T(t)|f|$.  The boundedness of $T(t)$ ensures boundedness of $S(t)$, 
which by convexity implies the Lipschitz continuity of $S(t)$ that allows for an extension to the Orlicz 
heart $M^\Phi$, see Theorem~\ref{thm:ext} and Theorem~\ref{thm:gexp ext}. 

To link the extended semigroup $S(t)_{t \ge 0}$ with the viscous Hamilton-Jacobi equation,  its generator 
$A$ has to be identified with the right-hand side of equation~\eqref{eq:CP}. However,  determining the 
generator on the whole domain $D(A)$ seems to be rather difficult.  We therefore focus on the (symmetric) 
Lipschitz set, which is also invariant under the semigroup,  see~\cite[Section~5]{BK20}.  The Lipschitz set 
consists of all functions $f$ such that the mapping $t\mapsto S(t)f$ is not only continuous,  but Lipschitz 
continuous. The Lipschitz set is crucial for the construction of $(S(t))_{t\geq 0}$ on $C_0$, because it 
enables us to extend the semigroup from the dyadic numbers to all $t\geq 0$.  We explicitly determine the 
symmetric Lipschitz set as the domain of the Laplacian in $L^\infty$ restricted to $C_0\cap M^\Phi$.  Then, 
on the domain $D(A)$ restricted to the symmetric Lipschitz set, the generator is given by 
$Af=\frac{1}{2}\Delta f+H(\nabla f)$,  where $\Delta f$ and $\nabla f$ are distributional derivatives. This allows 
to solve the Cauchy problem~\eqref{eq:CP} and estimate the solution,see Theorem~\ref{thm:reg}.

Our approach is different to the established theory of PDEs, and can be placed in context as follows. For 
$H(x)=|x|^q$ with $q>0$, and bounded continuous initial data,  existence and uniqueness of classical 
solutions in H\"older spaces have been established in~\cite{GGK2003}.  
The case $H(x)=a|x|^q$ in $L^p$-spaces is studied in~\cite{BSW2002}.  The focus is on the study of mild solutions,
which are also classical solutions by the theory of parabolic equations.  The choice of the initial data depends 
on whether $q<2$ or $q\geq 2$ and $a<0$ or $a>0$.  We want to point out that existence and uniqueness
can both fail if $a>0$, $q\geq 2$ or $p$ is to small.  As mentioned before,  in~\cite{FPR2001} the authors
made the assumption that the initial value is exponentially integrable.  More precisely, 
the existence of a weak solution in $L^p$ on a bounded domain $\Omega$ is shown, as well as the implication
\[ \int_\Omega e^{c|f(x)|}\,\d x<\infty \;\Longrightarrow\; 
	\sup_{t\in [0,T]}\int_\Omega e^{c|u(t,x)|}\,\d x<\infty \]
for a suitable constant $c\geq 0$.  This is similar to our statement $S(t)\colon M^\Phi\to M^\Phi$.  
Equations with degenerate coercivity and quadratic gradient term have been studied in~\cite{DGLS2006}. 
For non-convex Hamiltonians $H$,  existence and uniqueness of viscosity solutions has been 
investigated in~\cite{Davini19}.  There, for very regular initial data,  it is also shown how one can obtain
smooth solutions from the classical theory in~\cite{LSU1968}.  For long time behaviour and the 
question of convergence to a stationary solution,  we refer to~\cite{Gilding2005,SZ2006}.  
Concerning the regularity of solutions,  we want to mention both H\"older and Lipschitz estimates
for viscosity and weak solutions~\cite{AT2015,CS2012,CG2020a} as well as maximal
$L^p$-regularity for functions defined on the torus~\cite{CG2020}.  
From a stochastic point of view, the study of quadratic backward differential equations
leads to PDEs with quadratic gradient terms,  see~\cite{DHX2011,CPX2017,Kobylanski2000,BH2006}.
In particular,  the solution of the PDE has a stochastic representation similar to equation~\eqref{eq:square}
or the fundamental solution of the linear heat equation.\\

The paper is organized as follows.  In Section~\ref{sec:Oheart},  we collect some basic concepts and 
results in Orlicz hearts. In Section~\ref{sec:theory}, we study the extension of convex semigroups to 
Orlicz hearts and analyze their main properties.  In Section~\ref{sec:VHJ},  we apply these results to 
the viscous Hamilton-Jacobi equation.

\section{Preliminaries on Orlicz spaces}
\label{sec:Oheart}

Fix $d\in\N$, and denote by $\lambda$ the Lebesgue measure on $\R^d$. Let $L^0:=L^0(\R^d;\R)$ 
be the set of all Borel measurable functions $f\colon\R^d\to\R$, where two of them are identified if 
they coincide $\lambda$-almost everywhere (a.e.).  On $L^0$ we consider the pointwise partial order 
defined as $f\leq g$ if and only if  $f(x)\leq g(x)$ for $\lambda$-almost every $x\in\R^d$.  Furthermore,  
if the integral is well-defined, we write
\[ \int_{\R^d}f\,\d\lambda:=\int_{\R^d}f(x)\, \d x:=\int_{\R^d}f(x)\, \lambda(\d x).  \]
Throughout this section, let $\Phi\colon\R\to\R_+$ be a Young function.  Here, we follow the definition 
in~\cite{Rao},  i.e.,  $\Phi$ is convex,  $\Phi(0)=0$,  $\Phi(-x)=\Phi(x)$ and $\lim_{x\to\infty}\Phi(x)=\infty$. 
Furthermore, we assume that $\Phi(x)=0$ if and only if $x=0$. The corresponding Orlicz heart is defined as
\[ M^\Phi:=\left\{f\in L^0\colon \int_{\R^d}\Phi\left(\frac{f}{m}\right)\d\lambda<\infty\mbox{ for all } m>0\right\}.\]
We endow $M^\Phi$ with the Luxemburg norm 
\[ \|f\|_\Phi:=\inf\left\{m>0\colon \int_{\R^d}\Phi\left(\frac{f}{m}\right)\d\lambda\leq 1\right\}.  \]
Then,  $(M^\Phi,\|\cdot\|_\Phi , \leq)$ is a Banach lattice.  Moreover, for every sequence 
$(f_n)_{n\in\N}\subset M^\Phi$ and $f\in M^\Phi$ such that
\begin{equation} \label{eq:conv}
	\lim_{n\to\infty}\int_{\R^d}\Phi\left(\frac{f-f_n}{m}\right)\d\lambda=0 \quad\mbox{for all } m\in (0,1],
\end{equation}
it follows directly from the definition of the norm that  $\lim_{n\to\infty}\|f-f_n\|_\Phi=0$.

We frequently work with the following versions of the Luxemburg norm on $M^\Phi$, which for each 
$R\geq 1$ are defined as 
\[ \|f\|_{\Phi,R}=\inf\left\{m>0\colon\int_{\R^d}\Phi\left(\frac{f}{m}\right)\d\lambda\leq R\right\}.  \]
It is a straightforward application of the monotone convergence theorem that for each 
$f\in M^\Phi\backslash\{0\}$, the infimum in the previous equation is attained at $\|f\|_{\Phi,R}$.
The subsequent result shows that all norms $\|\cdot\|_{\Phi,R}$ are equivalent. In particular, 
$(M^\Phi,\|\cdot\|_{\Phi,R},\leq)$ is a Banach lattice for all $R\geq 1$.

\begin{lemma} \label{lem:norm}
 It holds $\|f\|_{\Phi,R}\leq\|f\|_\Phi\leq R\|f\|_{\Phi,R}$ for all $f\in M^\Phi$ and $R\geq 1$. 
\end{lemma}
\begin{proof}
 For $f=0$, the statement is obvious. So, let $f\in M^\Phi\backslash\{0\}$ and $R\geq 1$.  On the one 
 hand, by definition of the norms, we have $\|f\|_{\Phi,R}\leq\|f\|_\Phi$. On the other hand, since $\Phi$ 
 is convex and $\Phi(0)=0$,  we obtain
 \[ \int_{\R^d}\Phi\left(\frac{f}{R\|f\|_{\Phi,R}}\right)\d\lambda
 	\leq\frac{1}{R}\int_{\R^d}\Phi\left(\frac{f}{\|f\|_{\Phi,R}}\right)\d\lambda\leq 1.  \]
 This shows that $\|f\|_\Phi\leq R\|f\|_{\Phi,R}$.
\end{proof}

Let $B_R(f,r):=\{g\in M^\Phi\colon\|g-f\|_{\Phi,R}\leq r\}$ denote the closed ball around $f\in M^\Phi$ 
with radius $r\geq 0$. Furthermore, set $B_\Phi(f,r):=B_1(f,r)$ and $B_R(r):=B_R(0,r)$.

\begin{remark} \label{rem:union}
 For every $r>0$, we have 
 \[ M^\Phi=\bigcup_{R\geq 1} B_R(r).  \]
 Indeed, let $f\in M^\Phi$ and $r>0$.  By definition of the modified Luxembourg norm and the Orlicz 
 heart,  we have $f\in B_R(r)$ for 
 \[ R:=1+\int_{\R^d}\Phi\left(\frac{f}{r}\right)\d\lambda<\infty.  \]
\end{remark}

Let $C_c^\infty:=C_c^\infty(\R^d;\R)$ be the set of all infinitely differentiable functions 
$f\colon\R^d\to\R$ with compact support.

\begin{lemma} \label{lem:molify}
 Let $\eta\in C_c^\infty$ with $\eta\geq 0$ and $\int_{\R^d}\eta \,\d\lambda =1$.  Then, it holds 
 $f*\eta\in M^\Phi$ and $\|f*\eta\|_{\Phi,R}\leq\|f\|_{\Phi,R}$ for all $f\in M^\Phi$ and $R\geq 1$. 
\end{lemma}
\begin{proof}
 First, we show $f*\eta\in M^\Phi$.  Let $f\in M^\Phi$ and $m>0$.  We apply Jensen's inequality 
 on the probability measure $\eta\,\d\lambda$,  Fubini's theorem and the transformation theorem 
 for the Lebesgue measure to estimate
 \begin{align}
  &\int_{\R^d}\Phi\left(\frac{f*\eta}{m}\right)\d\lambda
  =\int_{\R^d}\Phi\left(\int_{\R^d}\frac{f(x-y)}{m}\eta(y)\,\d y\right)\d x \nonumber \\
  &\leq\int_{\R^d}\int_{\R^d}\Phi\left(\frac{f(x-y)}{m}\right)\eta(y)\,\d y\,\d x 
  =\int_{\R^d}\left(\int_{\R^d}\Phi\left(\frac{f(x-y)}{m}\right)\d x\right)\eta(y)\,\d y \nonumber \\
  &=\int_{\R^d}\left(\int_{\R^d}\Phi\left(\frac{f(x)}{m}\right)\d x\right)\eta(y)\,\d y 
  =\int_{\R^d}\Phi\left(\frac{f}{m}\right)\d\lambda<\infty.  \label{eq:molify}
 \end{align}
 
 Second,  it follows from the definition of the modified Luxemburg norm and inequality~\eqref{eq:molify} 
 that $\|f*\eta\|_{\Phi,R}\leq\|f\|_{\Phi,R}$ for all $R\ge 1$.
\end{proof}

We endow $\R^d$ with the Euclidean norm $|\cdot|$ and denote by $B_{\R^d}(r):=\{x\in\R^d\colon |x|\leq r\}$ 
the closed ball around the origin with radius $r\ge 0$.

\begin{lemma} \label{lem:approx}
 For every $R\geq 1$, $r\geq 0$ and $f\in B_R(r)$,  there exists a sequence $(f_n)_{n\in\N}$ in 
 $C_c^\infty(\R^d)\cap B_R(r)$ with $\lim_{n\to\infty}\|f-f_n\|_\Phi=0$.  In particular, 
 $C_c^\infty\subset M^\Phi$ is dense. 
\end{lemma}
\begin{proof}
 Fix $R\geq 1$ and $r\geq 0$.  First,  let $f\in B_R(r)$ be arbitrary.  Then, for every $k\in\N$,  the function 
 $f_k:=f\one_{\{|f|\leq k\}\cap B_{\R^d}(k)}$ is bounded and compactly supported.   Moreover,  we have
 $|f_k|\leq |f|$ for all $k\in\N$ and $f_k\to f$ $\lambda$-a.e.  as $k\to\infty$.  Hence,  
 $(f_k)_{k\in\N}\subset B_R(r)$, and by~\cite[Theorem~14 in Chapter~3.4]{Rao}, we obtain 
 $\|f-f_k\|_\Phi\to 0$.  
 
 Second, let $f\in B_R(r)$ be bounded and compactly supported.  Let $\eta\in C_c^\infty(\R^d)$ 
 be a molifier, i.e.,  $\eta\geq 0$ and $\int_{\R^d}\eta\,\d\lambda=1$.  Define $\eta_n(x):=n^d\eta(nx)$ 
 for all $n\in\N$ and $x\in\R^d$.  Let $f_n:=f*\eta_n\in C_c^\infty(\R^d)$ for all $n\in\N$. By 
 Lemma~\ref{lem:molify}, we have $(f_n)_{n\in\N}\subset B_R(r)$. Since $f$ is bounded and compactly 
 supported, we obtain \eqref{eq:conv} by the dominated convergence theorem, and therefore 
 $\|f-f_n\|_\Phi\to 0$. 
\end{proof}

\section{Semigroups on Orlicz hearts}
\label{sec:theory}

Denote by $\R_+:=\{x\in\R\colon x\geq 0\}$ the positive real numbers including zero. Throughout this 
section,  let $\phi\colon\R_+\to\R_+$ be a convex, strictly increasing, twice continuously differentiable
function which satisfies $\lim_{x\to\infty}\phi(x)=\infty$.  It follows that the inverse function 
$\phi^{-1}\colon [\phi(0),\infty)\to\R_+$ is concave and strictly increasing.  
Furthermore,  $\Phi\colon\R\to\R_+,\; x\mapsto\phi(|x|)-\phi(0)$ is a Young function satisfying $\Phi(x)=0$ 
if and only if $x=0$.  We follow the notations from Section~\ref{sec:Oheart}, and set 
$M^\Phi_+:=\{f\in M^\Phi\colon f\geq 0\}$.

Let $(X_t)_{t\geq 0}$ be a $d$-dimensional L\'evy process on a probability space $(\Omega,\F,\P)$, i.e., 
a stochastic process with stationary and independent increment, taking values in $\R^d$,  and $X_0=0$ 
$\P$-almost surely.  For an overview on L\'evy processes,  we refer to~\cite{applebaum2009} and~\cite{sato1999}.  
We denote the expectation of a random variable $X\colon\Omega\to\R$ by
\[ \E[X]:=\int_\Omega X(\omega)\,\P(d\omega),  \]
whenever the right hand side is well-defined.  From now on, we work under the following assumption.

\begin{assumption} \label{ass:OSG}
 Suppose that $\Phi$ and $X$ satisfy the following properties:
 \begin{enumerate}
 \item For all $t\geq 0$, the distribution $\P\circ X_t^{-1}$ is absolutely continuous w.r.t.  the Lebesgue measure 
  $\lambda$ with essentially bounded Radon-Nikodym derivative.
 \item For every $c\geq 0$, there exists $r\geq 0$ such that 
  $\lim_{t\downarrow 0}\P(|X_t|\geq r)\Phi\big(\frac{c}{t}\big)=0$.
 \end{enumerate}
\end{assumption}

Assumption~\ref{ass:OSG} is illustrated by the following example.

\begin{example}
 Let $(X_t)_{t\geq 0}$ be a $d$-dimensional Brownian motion.  For every $t\geq 0$, let $\n(0,t\one)$ be 
 the $d$-dimensional normal distribution with mean zero and variance $t\one$,  where $\one\in\R^{d\times d}$
 denotes the identity matrix. The distribution $\P\circ X_t^{-1}=\n(0,t\one)$ satisfies Assumption~\ref{ass:OSG}(i) 
 for all $t\geq 0$.  
 Moreover, Assumption~\ref{ass:OSG}(ii) holds e.g.  for  
 \begin{align*}
  \phi(x) :=x^p \;\mbox{ for }\; p\in [1,\infty),  \quad
  \phi(x) :=e^{x},  \quad
  \phi(x) :=(bx-1)e^{bx}+1 \;\mbox{ for }\; b\geq 0. 
 \end{align*}
 The last choice of $\phi$ will be used in Section~\ref{sec:VHJ}.  There,  we also verify Assumption~\ref{ass:OSG}(ii). 
\end{example}

\subsection{Extension of semigroups} 
\label{sec:ext}

Let $C_0:=C_0(\R^d;\R)$ be the space of all continuous functions $f\colon\R^d\to\R$ which vanish at 
infinity, i.e., $\lim_{|x|\to\infty}|f(x)|=0$. As usual, $C_0$ is endowed with the supremum norm 
$\|f\|_\infty:=\sup_{x\in\R^d}|f(x)|$.

From now on, let $(S(t))_{t\geq 0}$ be a strongly continuous,  convex semigroup on $C_0$, i.e., a family of 
convex operators $S(t)\colon C_0\to C_0$ which satisfies
\begin{enumerate}
 \item $S(0)=\id_{C_0}$,
 \item $S(s+t)=S(s)S(t)$ for all $s,t\ge 0$,
 \item $\lim_{t\downarrow 0} \|S(t)f-f\|_\infty=0$ for all $f\in C_0$.
\end{enumerate}
The respective generator $A_\infty\colon D(A_\infty)\to C_0$ is defined as 
\[ A_\infty f :=\lim_{t\downarrow 0}\frac{S(t)f-f}{t},  \]
where the limit is understood w.r.t.~the supremum norm, and the domain $ D(A_\infty)$ is the set of all 
$f\in C_0$ for which this limit exists.  Our goal is to extend the semigroup  $(S(t))_{t\geq 0}$ from $C_0$ 
to the Orlicz heart $M^\Phi$. To that purpose, we assume that 
\begin{equation} \label{eq:ST}
 |S(t)f|\leq T(t)|f| \quad\mbox{for all } t\geq 0 \mbox{ and } f\in C_0\cap M^\Phi, 
\end{equation}
for a dominating family $(T(t))_{t\geq 0}$ of operators $T(t)\colon M^\Phi_+\to M^\Phi_+$. For the 
remainder of the section, we fix $a\geq 0$, and focus on the dominating family of the form
\begin{equation} \label{eq:def OSG}
 \big(T(t)f\big)(x):=\phi^{-1}\Big(\E\big[\phi(e^{at}f(x+X_t))\big]\Big)
\end{equation}
for all $t\geq 0$,  $f\in M^\Phi_+$ and $x\in\R^d$.

\begin{remark}
 For $a=0$,  the family $(T(t))_{t\geq 0}$ itself is a semigroup.  In particular, for the choice $\phi(x)=e^{x}$, 
 we obtain the entropic semigroup 
 \[ (T(t)f)(x)=\log\big(\E[\exp(f(x+X_t))].\big)  \]
 It follows from It\^o's formula that for $f$ sufficiently smooth, the function $u(t,\cdot):=T(t)f$ solves 
 the Cauchy problem $\partial_t u=\frac{1}{2}\Delta u+\frac{1}{2}|\nabla u|^2$ for $t>0$, $u(0,\cdot)=f$.
\end{remark}

The following lemma is the same as~\cite[Proposition~2.44]{Foellmer}. Recall that in decision theory, 
the expression $-\tfrac{u''(x)}{u'(x)}$ is called the Arrow-Pratt measure of absolute risk-aversion
of a utility function $u$.

\begin{lemma} \label{lem:ap}
 Let $u,v\colon\R_+\to\R_+$ be two strictly increasing, twice continuously differentiable functions.  Assume
 \begin{equation} \label{eq:ap}
  \frac{u''(x)}{u'(x)}\leq\frac{v''(x)}{v'(x)} \quad\mbox{for all } x\in\R_+.
 \end{equation}
 Then, for every random variable $X\colon\Omega\to\R_+$, it holds $u^{-1}\big(\E[u(X)]\big)\leq v^{-1}\big(\E[v(X)]\big)$. 
\end{lemma}
\begin{proof}
 Define $F:=v\circ u^{-1}$.  Let $x\in\R_+$ and $y:=u^{-1}(x)$.  It follows from $v'(y)>0$ and inequality~\eqref{eq:ap}
 that 
 \[ F''(x)=\frac{v'(y)}{u'(y)^2}\left(\frac{v''(y)}{v'(y)}-\frac{u''(y)}{u'(y)}\right)\geq 0.  \]
 This shows that $F$ is convex.  Hence, Jensen's inequality implies
 \[ u^{-1}\big(\E[u(X)]\big)=v^{-1}\big(F\big(\E[u(X)]\big)\big)\leq v^{-1}\big(\E[F(u(X))|\big)=v^{-1}\big(\E[v(X)]\big).  
 	\qedhere \]
\end{proof}

\begin{corollary} \label{cor:ap}
 For every random variable $X\colon\Omega\to\R_+$, it holds
 \[ c\phi^{-1}\big(\E[\phi(X)]\big)\leq \phi^{-1}\big(\E[\phi(cX)|\big) \quad\mbox{for all } c\geq 1.  \]
\end{corollary}
\begin{proof}
 Apply Lemma~\ref{lem:ap} with $u(x):=\phi(x)$ and $v(x):=\phi(cx)$ for all $x\in\R_+$. 
\end{proof}

Next, we investigate basic properties of the dominating family $(T(t))_{t\geq 0}$.  Denote by 
$C_c:=C_c(\R^d;\R)$ the set of all continuous functions $f\colon\R^d\to\R$ with compact support, 
and define $C_c^+:=\{f\in C_c\colon f\geq 0\}$.

\begin{theorem} \label{thm:OSG}
 The family $(T(t))_{t\geq 0}$ satisfies the following properties:
 \begin{enumerate}
  \item $T(t)\colon M^\Phi_+\to M^\Phi_+$ for all $t\geq 0$. 
  \item $\|T(t)f\|_{\Phi,R}\leq e^{at}\|f\|_{\Phi,R}$ for all $R\geq 1$ and $f\in M^\Phi_+\cap B_R(e^{-at})$. 
  \item $T(0)=\id_{M^\Phi_+}$ and $T(s)T(t)f\leq T(s+t)f$ for all $s,t\geq 0$ and $f\in M^\Phi_+$.  
  \item For every $f\in C_c^+$ and $m\in (0,1]$, there exists $r\geq 0$ such that 
   \begin{equation} \label{eq:tight}
    \lim_{t\downarrow 0}\int_{B(0,r)^c}\Phi\left(\frac{T(t)f}{mt}\right)\d\lambda=0.  
   \end{equation}
 \end{enumerate}
\end{theorem}
\begin{proof}
 First, we show $T(t)\colon M^\Phi_+\to M^\Phi_+$ for all $t\geq 0$.  Fix $f\in M^\Phi_+$.  
 It follows from the definition of $M^\Phi$, Assumption~\ref{ass:OSG}(i), $\phi=\Phi+\phi(0)$ 
 and $\int_{\R^d}\phi(0)\,\d(\P\circ X_t^{-1})=\phi(0)$ that
 \[ \E\big[\phi(e^{at}f(x+X_t))\big]
 	=\int_{\R^d}\phi\big(e^{at}f(x+y)\big)\big(\P\circ X_t^{-1}\big)(\d y)<\infty
 	\quad\mbox{for all } x\in\R^d.  \]
 Thus, $T(t)f\colon\R^d\to\R$ is a well-defined function.  By Tonelli's theorem, the function 
 $T(t)f$ is measurable.  Let $m>0$.  We distinguish between two cases.  On the one hand, 
 for $m\in (0,1]$, we use Corollary~\ref{cor:ap}, monotonicity of $\phi$,  Fubini's theorem, the 
 transformation theorem for the Lebesgue measure and the definition of $M^\Phi$ to estimate
 \begin{align}
  \int_{\R^d}\Phi\left(\frac{T(t)f}{m}\right)\d\lambda \nonumber
  &=\int_{\R^d}\phi\left(\frac{1}{m}\phi^{-1}\big(\E[\phi(e^{at}f(x+X_t))]\big)\right)-\phi(0)\, \d x \nonumber \\
  &\leq\int_{\R^d}\E\left[\phi\left(\frac{e^{at}f(x+X_t)}{m}\right)-\phi(0)\right]\d x \nonumber \\
  &=\E\left[\int_{\R^d}\Phi\left(\frac{e^{at}f(x+X_t)}{m}\right)\d x\right]
  =\int_{\R^d}\Phi\left(\frac{e^{at}f(x)}{m}\right)\d x<\infty.  \label{eq:T norm}
 \end{align}
On the other hand, for $m\geq 1$, it follows from monotonicity of $\Phi$ that
 \[ \int_{\R^d}\Phi\left(\frac{T(t)f}{m}\right)\d\lambda\leq\int_{\R^d}\Phi(T(t)f)\,\d\lambda
 	=\int_{\R^d}\Phi(e^{at}f)\,\d\lambda<\infty.  \]
 
 Second,  let $R\geq 1$ and $f\in M^\Phi_+\cap B_R(e^{-at})$.  If $f=0$,  then $T(t)f=0$.  Thus,
 assume $\|f\|_{\Phi,R}>0$.  For $m:=e^{at}\|f\|_{\Phi,R}\in (0,1]$, it follows from 
 inequality~\eqref{eq:T norm} that
 \[ \int_{\R^d}\Phi\left(\frac{T(t)f}{m}\right)\d\lambda
 	\leq\int_{\R^d}\Phi\left(\frac{f}{\|f\|_{\Phi,R}}\right)\d\lambda\leq R.  \]
 This shows $\|T(t)f\|_{\Phi,R}\leq e^{at}\|f\|_{\Phi,R}$.  
 
 Third, let $f\in M^\Phi_+$, $x\in\R^d$, and $s,t\geq 0$.  For all $t\geq 0$, let $\F_t:=\sigma(X_s\colon [0,t])$ 
 be the $\sigma$-Algebra generated by L\'evy process up time $t$.  Corollary~\ref{cor:ap},  the stationary
 and independent increments of $(X_t)_{t\geq 0}$, monotonicity of $\phi^{-1}$ and the tower property yield
 \begin{align*}
  \big(T(s)T(t)f\big)(x)
  &=\phi^{-1}\Big(\E\big[\phi\big(e^{as}(T(t)f)(x+X_s)\big)\big]\Big) \\
  &=\phi^{-1}\left(\E\left[\phi\Big(e^{as}\phi^{-1}\Big(\E\big[\phi\big(e^{at}f(y+ X_t)\big)\big]\Big)\Big)
  	\Big|_{y=x+X_s}\right]\right)\\
  &\leq\phi^{-1}\Big(\E\Big[\E\big[\phi\big(e^{a(s+t)}f(x+X_s+(X_{s+t}-X_s))\big)\big|\F_s\big]\Big]\Big) \\
  &=\big(T(s+t)f\big)(x).
 \end{align*}  
 Moreover, $X_0=0$ $\P$-a.s. implies $T(0)=\id_{M^\Phi_+}$.
 
 Fourth, let $f\in C_c^+$ and $m\in (0,1]$.  Choose $r_0\geq 0$ with $\supp(f)\subset B(r_0):=B_{\R^d}(r_0)$. 
 Define
 \begin{equation} \label{eq:OSG c}
  c:=\max\left\{\lambda\big(B(r_0)\big),\frac{e^a\|f\|_\infty}{m}\right\}.  
 \end{equation}
 By Assumption~\ref{ass:OSG}(ii), there exists $r\geq \max\{c,r_0\}$ such that
 \begin{equation} \label{eq:OSG lim}
  \lim_{t\downarrow 0}\P(|X_t|\geq r)\Phi\left(\frac{c}{t}\right)=0. 
 \end{equation}
 Let $t\in (0,1]$.  We use inequality~\eqref{eq:T norm}, $\supp(f)\subset B(r_0)\subset B(r)$,  $\Phi(0)=0$,  
 Fubini's theorem, $\Phi\geq 0$, the transformation theorem for the Lebesgue measure,  equation~\eqref{eq:OSG c}
 and equation~\eqref{eq:OSG lim} to estimate
 \begin{align*}
  &\int_{B(2r)^c}\Phi\left(\frac{T(t)f}{mt}\right)\d\lambda
  \leq\int_{B(2r)^c}\E\left[\Phi\left(\frac{e^{at}f(x+X_t)}{mt}\right)\right]\d x \\
  &= \int_{B(2r)^c}\E\left[\Phi\left(\frac{e^{at}f(x+X_t)}{mt}\right)\one_{\{|X_t|\geq r\}}\right] \d x  \\
  &=\E\left[\left(\int_{B(2r)^c}\Phi\left(\frac{e^{at}f(x+X_t)}{mt}\right)\d x\right)\one_{\{|X_t|\geq r\}}\right] \\
  &\leq\E\left[\left(\int_{\R^d}\Phi\left(\frac{e^{at}f(x+X_t)}{mt}\right)\,dx\right)\one_{\{|X_t|\geq r\}}\right] 
  =\E\left[\left(\int_{\R^d}\Phi\left(\frac{e^{at}f(x)}{mt}\right)\d x\right)\one_{\{|X_t|\geq r\}}\right] \\
  &=\P(|X_t|\geq r)\int_{\R^d}\Phi\left(\frac{e^{at}f}{mt}\right)\d\lambda 
  =\P(|X_t|\geq r)\int_{B(r_0)}\Phi\left(\frac{e^{at}f}{mt}\right)\d\lambda \\
  &\leq c\P(|X_t|\geq r)\Phi\left(\frac{c}{t}\right) \to 0 \quad\mbox{as } t\downarrow 0. \qedhere
 \end{align*}
\end{proof}

Now, we are ready to state our main extension result.

\begin{theorem} \label{thm:ext}
 Suppose that Assumption~\ref{ass:OSG} holds.  Let $(S(t))_{t\geq 0}$ be a strongly continuous, 
 convex semigroup on $C_0$ which satisfies inequality~\eqref{eq:ST} with the dominating family 
 $(T(t))_{t\geq 0}$ given by equation~\eqref{eq:def OSG}. Then,  the following statements hold:
 \begin{enumerate}
  \item $S(t)\colon C_0\cap M^\Phi\to C_0\cap M^\Phi$ for all $t\geq 0$. 
  \item $\|S(t)f\|_{\Phi,R}\leq e^{at}\|f\|_{\Phi,R}$ for all $t\geq 0$, $R\geq 1$ and $f\in C_0\cap B_R(e^{-at})$.
  \item $\|S(t)f-S(t)g\|_{\Phi,R}\leq 4e^{at}\|f-g\|_{\Phi,R}$ for all $t\geq 0$, $R\geq 1$,\,
   $f,g\in C_0\cap B_R\big(e^{-at}/3\big)$. 
  \item For every $f\in M^\Phi$ and $t_0\geq 0$, there exist $R\geq 1$ and $r>0$ such that
   \begin{equation} \label{eq:ext lip}
    \sup_{t\in [0,t_0]}\|S(t)g-S(t)h\|_\Phi\leq 4Re^{at}\|g-h\|_\Phi 
    \mbox{ for all } g,h\in C_0\cap B_\Phi(f,r). 
   \end{equation}
 \end{enumerate}
 Hence, for every $t\geq 0$, the operator $S(t)$ has a unique continuous extension 
  \[ \tilde{S}(t)\colon M^\Phi\to M^\Phi \]
 satisfying the properties (ii)-(iv) without the restriction to functions in $C_0$. Furthermore,
 $(\tilde{S}(t))_{t\geq 0}$ is a strongly continuous,  convex semigroup on $M^\Phi$.  For all 
 $f\in D(A_\infty)\cap C_c$ with $A_\infty f\in C_c$,  it holds
 \[ \lim_{t\downarrow 0}\left\|\frac{S(t)f-f}{t}-A_\infty f\right\|_\Phi=0. \]
\end{theorem}
\begin{proof}
 First,  we show $S(t)f\colon C_0\cap M^\Phi\to C_0\cap M^\Phi$.  Fix $t\geq 0$ and 
 $f\in C_0\cap M^\Phi$. From inequality~\eqref{eq:ST} and Theorem~\ref{thm:OSG}(i),  we obtain
 \[ \int_{\R^d}\Phi\left(\frac{S(t)f}{m}\right)\d\lambda
 	\leq\int_{\R^d}\Phi\left(\frac{T(t)|f|}{m}\right)\d\lambda<\infty \quad\mbox{for all } m>0.  \]
   
 Second, let $t\geq 0$, $R\geq 1$ and $f\in C_0\cap B_R(e^{-at})$.  Inequality~\eqref{eq:ST} and 
 Theorem~\ref{thm:OSG}(ii) imply $\|S(t)f\|_{\Phi,R}\leq\|T(t)(|f|)\|_{\Phi,R}\leq e^{at}\|f\|_{\Phi,R}$.  
 
 Third,  let $t\geq 0$, $R\geq 1$ and $f,g\in C_0\cap B_R\big(e^{-at}/3\big)$.  Define
 \[ S_f(t)\colon C_0\to C_0,\; h\mapsto S(t)(f+h)-S(t)f.  \]
 The operator $S_f(t)$ is convex and $S_f(t)0=0$.  Moreover, by~(ii),  it holds
 \[ \|S_f(t)h\|_{\Phi,R}\leq\tfrac{4}{3} 
 	\quad\mbox{for all } h\in C_0\cap B_R\left(0,\tfrac{2e^{-at}}{3}\right).  \]
 Hence,~\cite[Lemma~A.1]{DKN19} implies
 \[ \|S_f(t)h\|_{\Phi,R}\leq 4e^{at}\|h\|_{\Phi,R} 
 	\quad\mbox{for all } h\in C_0\cap B_R\left(0,\tfrac{2e^{-at}}{3}\right).  \]
 Applying the previous inequality for $h:=g-f$, we obtain
 \[ \|S(t)f-S(t)g\|_{\Phi,R}=\|S_f(t)(g-f)\|_{\Phi,R}\leq 4e^{at}\|f-g\|_{\Phi,R}.  \]
 
 Forth,  let $f\in M^\Phi$ and $t_0\geq 0$.  By Lemma~\ref{lem:norm} and Remark~\ref{rem:union}, 
 there exists $R\geq 1$ and $r>0$ with $B_\Phi(f,r)\subset B_R(f,r)\subset B_R\big(e^{-at_0}/3\big)$.  
 Hence,~(iv) follows from~(iii) and Lemma~\ref{lem:norm}. 
 
 Fifth,  we extend $(S(t))_{t\geq 0}$.  To that end,  let $t\geq 0$ and $f\in M^\Phi$.  Choose $R\geq 1$ 
 and $r>0$ such that inequality~\eqref{eq:ext lip} holds for $t_0=t$.  By Lemma~\ref{lem:approx},  there 
 exists a sequence $(f_n)_{n\in\N}$ in $C_0\cap B_\Phi(f,r)$ such that $\|f-f_n\|_\Phi\to 0$ as $n\to\infty$. 
 Define $\tilde{S}(t)f:=\lim_{n\to\infty}S(t)f_n$.  By inequality~\eqref{eq:ext lip}, the limit exists and does
 not depend on the choice of the  sequence $(f_n)_{n\in\N}$.  Moreover, $\tilde{S}(t)$ is the unique 
 continuous extension of $S(t)$.  It is straightforward to verify that properties~(ii)-(iv) and convexity 
 are preserved in the limit. To show the semigroup property,  let $s,t\geq 0$ and $f\in M^\Phi$.  Choose 
 a sequence $(f_n)_{n\in\N}$ in $C_0\cap M^\Phi$ with $\|f-f_n\|_\Phi\to 0$ as $n\to\infty$.  By
 inequality~\eqref{eq:ext lip},  there exist $R\geq 1$ and $r>0$ such that 
 \[ \big\|\tilde{S}(s)g-\tilde{S}(s)h\big\|_\Phi\leq 4Re^{as}\|g-h\|_\Phi 
 	\quad\mbox{for all } g,h\in B_\Phi\big(\tilde{S}(t)f,r\big).  \]
 Since $S(t)f_n\to\tilde{S}(t)f$ as $n\to\infty$,  we can assume that $S(t)f_n\in B_\Phi(\tilde{S}(t)f,r)$
 for all $n\in\N$.  Then, the semigroup property of $(S(t))_{t\geq 0}$ on $C_0$ implies
 \begin{align*}
  \big\|\tilde{S}(s)\tilde{S}(t)f-\tilde{S}(s+t)f\big\|_\Phi
  &\leq\big\|\tilde{S}(s)\tilde{S}(t)f-S(s)S(t)f_n\big\|_\Phi+\big\|S(s+t)f_n-\tilde{S}(s+t)f\big\|_\Phi \\
  &\leq 4Re^{as}\big\|\tilde{S}(t)f-S(t)f_n\big\|_\Phi+\big\|S(s+t)f_n-\tilde{S}(s+t)f\big\|_\Phi \\
  &\to 0 \quad\mbox{as } n\to\infty.  
 \end{align*}
 
 Sixth, we show that $(\tilde{S}(t))_{t\geq 0}$ is strongly continuous.  Let $f\in C_c$.  In order to prove
 $\lim_{t\downarrow 0}\|S(t)f-f\|_\Phi=0$,  it suffices to verify equation~\eqref{eq:conv}.  Fix $m\in (0,1]$.
 Choose $r\geq 0$ with $\supp(f)\subset B(r)$ such that inequality~\eqref{eq:tight} holds for $|f|$.  
 We have
 \[ \int_{\R^d}\Phi\left(\frac{S(t)f-f}{m}\right)\d\lambda
 	=\int_{B(r)}\Phi\left(\frac{S(t)f-f}{m}\right)\d\lambda +\int_{B(r)^c}\Phi\left(\frac{S(t)f}{m}\right)\d\lambda.  \]
 On the one hand, since $(S(t))_{t\geq 0}$ is strongly continuous on $C_0$ w.r.t.~the supremum norm, 
 we obtain
 \[ \lim_{t\downarrow 0}\int_{B(r)}\Phi\left(\frac{S(t)f-f}{m}\right)\d\lambda=0.  \]
 On the other hand,  inequality~\eqref{eq:ST} and inequality~\eqref{eq:tight} imply
 \[ \int_{B(r)^c}\Phi\left(\frac{S(t)f}{m}\right)\d\lambda \leq\int_{B(r)^c}\Phi\left(\frac{T(t)|f|}{mt}\right)\d\lambda 
 \to 0 \quad\mbox{as } t\downarrow 0.  \]
 We obtain $\lim_{t\downarrow 0}\|S(t)f-f\|_\Phi=0$ for all $f\in C_c$.  For arbitrary $f\in M^\Phi$, the claim
 follows by approximation from Lemma~\ref{lem:approx} and inequality~\eqref{eq:ext lip}.  
 
 Seventh,  let $f\in D(A_\infty)\cap C_c$ with $A_\infty f\in C_c$.  To prove the desired convergence, 
 we verify inequality~\eqref{eq:conv}.  Fix $m\in (0,1]$.  Choose $r\geq 0$ with $\supp(f),\,\supp(Af)\subset B(r)$
 such that equation~\eqref{eq:tight} holds for $|f|$. We have
 \begin{align*}
  &\int_{\R^d}\Phi\left(\frac{1}{m}\left(\frac{S(t)f-f}{t}\right)\right)\d\lambda \\
  &=\int_{B(r)}\Phi\left(\frac{1}{m}\left(\frac{S(t)f-f}{t}-Af\right)\right)\d\lambda
  	+\int_{B(r)^c}\Phi\left(\frac{S(t)f}{mt}\right)\d\lambda.
 \end{align*}
 On the one hand, since $f\in D(A_\infty)$,  we obtain
 \[ \lim_{t\downarrow 0}\int_{B(r)}\Phi\left(\frac{1}{m}\left(\frac{S(t)f-f}{t}-Af\right)\right)\d\lambda=0.  \]
 On the other hand, inequality~\eqref{eq:ST} and equation~\eqref{eq:tight} imply
 \[ \int_{B(r)^c}\Phi\left(\frac{S(t)f}{mt}\right)\d\lambda
 	\leq\int_{B(r)^c}\left(\frac{T(t)|f|}{mt}\right)\d\lambda \to 0 \quad\mbox{as } t\downarrow 0.  \qedhere \]
\end{proof}

\subsection{Convex semigroups on $M^\Phi$}

In this section, we show that the results from~\cite[Section~3]{DKN19} are applicable in our setting.
For notational simplification we write $(S(t))_{t\geq 0}$ for the extension $(\tilde{S}(t))_{t\geq 0}$. 
Recall that $(S(t))_{t\geq 0}$ is a strongly continuous, convex semigroup on $M^\Phi$ which is locally 
uniformly Lipschitz continuous, i.e., for every $f\in M^\Phi$ and $t_0\geq 0$, there exist $c\geq 0$ and 
$r>0$ such that 
\begin{equation} \label{eq:uniform lip}
 \sup_{t\in [0,t_0]}\|S(t)g-S(t)h\|_\Phi\leq c\|g-h\|_\Phi \quad\mbox{for all } g,h\in B_\Phi(f,r).  
\end{equation}
The respective generator $A\colon D(A)\to M^\Phi$  is defined as
\[ A f :=\lim_{t\downarrow 0}\frac{S(t)f-f}{t},\]
where the limit is understood w.r.t.~the Luxembourg norm, and the domain $ D(A)$ is the set of all 
$f\in M^\Phi$ for which this limit exists. It follows from~\cite[Theorem~14 in Chapter~3.4]{Rao}, that 
the Luxembourg norm is order continuous, i.e., for every net $(f_\alpha)_\alpha\subset M^\Phi$ with 
$f_\alpha\downarrow 0$,  we have $\|f_\alpha\|_\Phi\to 0$. In particular, $M^\Phi$ is Dedekind $\sigma$-
complete, see~\cite[Theorem~2.4.2]{meyer1991}. Hence, we can  apply the results 
from~\cite[Section~3]{DKN19}. There, the results are formulated for strongly continuous,  convex 
semigroups of bounded operators.  As a consequence of the convexity and the uniform boundedness 
principle, it follows from~\cite[Proposition~2.2 and Corollary~2.4]{DKN19} that these semigroups
are locally uniformly Lipschitz continuous. This is the crucial property (rather than boundedness) 
for the results in~\cite[Section~3]{DKN19}. A direct adaptation to the present setting is summarized 
in the following result.

\begin{theorem} \label{thm:PDE}
 The following statements hold:
 \begin{enumerate}
  \item $S(t)\colon D(A)\to D(A)$ for all $t\geq 0$. 
  \item Let $f\in D(A)$. Then, the mapping $S(\cdot)f\colon [0,\infty)\to M^\Phi$ is continuously
   differentiable. For every $t\geq 0$, the derivative is given by
   \[ \frac{d}{dt}\big(S(t)f\big)=\inf_{h>0}\frac{S(t)(f+hAf)-S(t)f}{h}
   	=\sup_{h<0}\frac{S(t)(f+hAf)-S(t)f}{h}.  \]
  \item The operator $A$ is closed, i.e.,  for every sequence $(f_n)_{n\in\N}\subset D(A)$
   with $f_n\to f$ and $A f_n\to g$ for some $f,g\in M^\Phi$, it holds $f\in D(A)$ with $A f=g$.  
  \item For every continuous function $v\colon[0,\infty)\to M^\Phi$ which satisfies
   \[ v(t)\in D(A)\quad\mbox{with}\quad\lim_{h\downarrow 0}\frac{v(t+h)-v(t)}{h}=Av(t) 
   	\quad\mbox{for all }t\geq 0, \]
  it holds $v(t)=S(t)f$ for $f:=v(0)$.
 \end{enumerate}
 In particular,  for every $f\in D(A)$,  the abstract Cauchy problem
 \begin{equation} \label{eq:CP2}
  \partial_t u(t)=Au(t) \;\mbox{ for all } t\geq 0,  \quad u(0)=f
 \end{equation}
 has a unique solution $u\in C^1([0,\infty);M^\Phi)\cap C([0,\infty);D(A))$. The solution is given 
 by $u(t):=S(t)f$ for all $t\geq 0$ and depends continuously on the initial data.  
  
\end{theorem}

\section{Viscous Hamilton-Jacobi equation}
\label{sec:VHJ}

In this section, we apply the results of Section~\ref{sec:theory} to the viscous Hamilton-Jacobi 
equation
\begin{equation} \label{eq:CP3}
 \begin{cases}
  \partial_t u(t,x)=\frac{1}{2}\Delta u(t,x)+H(\nabla u(t,x)),  & (t,x)\in (0,\infty)\times\R^d,  \\
  u(0,x)=f(x),  & x\in\R^d.
 \end{cases}
\end{equation}

Throughout this section, we make the following assumption.

\begin{assumption} \label{ass:H}
 The function $H\colon \R^d\to\R$ is convex and there exists $K\geq 0$ with
 \[ |H(x)|\leq K(|x|+|x|^2) \quad\mbox{for all }x\in\R^d.  \]
 Furthermore, there exists $r>0$ such that the convex conjugate 
 \[ L\colon\R^d\to [0,\infty],\; \lambda\mapsto\sup_{x\in\R^d}\big(\langle\lambda,x\rangle-H(x)\big) \]
 satisfies $\sup_{\lambda\in B(r)}L(\lambda)<\infty$.
\end{assumption}

\subsection{Construction of the semigroup}
\label{sec:sg}

Let $(W_t)_{t\geq 0}$ be a $d$-dimensional Brownian motion on a probability space $(\Omega,\F,\P)$. 
For every $t\geq 0$, $f\in C_0$ and $x\in\R^d$,  we define 
\[ \big(I(t)f\big)(x):=\sup_{\lambda\in\R^d}\big(\E[f(x+W_t+\lambda t)]-L(\lambda)t\big).  \]
It follows from Fenchel-Moreaus's theorem that 
$H(x)=\sup_{\lambda\in\R^d}\big(\langle x,\lambda\rangle-L(\lambda)\big)$ 
for all $x\in\R^d$.  Hence,  the family $(I(t))_{t\geq 0}$ has the desired derivative at $t=0$,  i.e.,
\[ \lim_{t\downarrow 0}\frac{I(t)f-f}{t}=\frac{1}{2}\Delta f+H(\nabla f)
	 \quad\mbox{for all } f\in C^2_0,  \]
where $C_0^2$ denotes the set of all twice differentiable functions $f\colon\R^d\to\R$ such that all 
partial derivatives up to order two are again elements of $C_0$.  To construct a semigroup 
$(S(t))_{t\geq 0}$ associated to the family $(I(t))_{t\geq 0}$,  we proceed as explained in the 
introduction.  Denote by $\T:=\{k2^{-n}\colon k,n\in\N_0\}$ the dyadic numbers.  For every $n\in\N$ 
and $t\in\T$,  define the partition $\pi_n^t:=\{k2^{-n}\colon k=0,\ldots 2^nt\}$ and the iterated 
operator $I(\pi_n^t):=I(2^{-n})^{2^nt}$. The following result is a direct application 
of~\cite[Theorem~6.2 and Theorem~6.3]{BK20}.

\begin{theorem} \label{thm:gexp C_0}
 There exists a family $(S(t))_{t\geq 0}$ of operators $S(t)\colon C_0\to C_0$ which satisfy the 
 following properties:
 \begin{enumerate}
  \item There exists a subsequence $(n_l)_{l\in\N}\subset\N$ such that 
   \[ \lim_{l\to\infty}\|S(t)f-I(\pi^t_{n_l})f\|_\infty=0 \quad\mbox{for all } t\in\T \mbox{ and } f\in C_0.  \]
  \item $S(0)=\id_{C_0}$ and $S(s)S(t)=S(s+t)$ for all $s,t\geq 0$.
  \item $S(t)$ is convex, monotone and $S(t)0=0$ for all $t\geq 0$. 
  \item $\|S(t)f-S(t)g\|_\infty\leq\|f-g\|_\infty$ for all $t\geq 0$ and $f,g\in C_0$.
  \item For every $f\in C_0$, the mapping $\R_+\to C_0,\; t\mapsto S(t)f$ is continuous. 
  \item For every $f\in C_0^2$,
   \[ \lim_{t\downarrow 0}\left\|\frac{S(t)f-f}{t}-A_\infty f\right\|_\infty =0,  \quad\mbox{where}\quad
   	A_\infty f:=\frac{1}{2}\Delta f+H(\nabla f).  \]
 \end{enumerate}
\end{theorem}
\begin{proof}
 We only have to verify the assumptions of~\cite[Subsection~6.1]{BK20}.  First,  it holds 
 \begin{equation} \label{eq:L growth}
  L(\lambda)\geq \tfrac{|\lambda|^2}{16K}\one_{B(2K)^c}(\lambda) \quad\mbox{for all } \lambda\in\R^d.  
 \end{equation} 
 Indeed,  this follows from Assumption~\ref{ass:H} by choosing $x:=\lambda/(4K)$ in the definition
 of the convex conjugate for all $|\lambda|\geq 2K$.  Inequality~\eqref{eq:L growth} implies
 $\lim_{|\lambda|\to\infty}L(\lambda)/|\lambda|=\infty$.  Second, since $H(0)=0$ and  $H$ is 
 sub-differentiable at $0$, there exists $\lambda_0\in\R^d$ with 
 $H(0)=\langle 0,\lambda_0\rangle-L(\lambda_0)$, and therefore $L(\lambda_0)=0$.  
\end{proof}

Next, we apply the results from Section~\ref{sec:theory}.  Define $b:=8K+1$ and 
\[ \phi\colon\R_+\to\R_+,\; x\mapsto (bx-1)e^{bx}+1.  \] 
The function $\phi$ is convex, strictly increasing, twice continuously differentiable with 
$\lim_{x\to\infty}\phi(x)=\infty$ and $\phi(0)=0$.  Define $\Phi\colon\R\to\R_+,\; x\mapsto\phi(|x|)$, 
so that $\Phi(x)=0$ if and only if $x=0$.  Choose $X_t:=W_t$ for all $t\geq 0$,  $a:=K^2$ and define 
$(T(t))_{t\geq 0}$ by equation~\eqref{eq:def OSG}, i.e. 
$\big(T(t)f\big)(x):=\phi^{-1}\big(\E[\phi(e^{at}f(x+X_t))]\big)$ for all $t\geq 0$,  $f\in M^\Phi_+$ 
and $x\in\R^d$.  In order to verify Assumption~\ref{ass:OSG}(ii), we need the following auxiliary result.

\begin{lemma} \label{lem:prob}
 There exist $t_0>0$ and $r_0\geq 0$ such that
  \[ \P(|W_t|\geq r)\leq te^{-\frac{r}{t}} \quad\mbox{for all } t\in [0,t_0] \mbox{ and } r\geq r_0.  \]
\end{lemma}
\begin{proof}
 Define $\S_u:=\{x\in\R^d\colon |x|=u\}$ for all $u\geq 0$ and let $|\S_1|$ be the area of the unit 
 sphere.  Fix $t>0$ and $r\geq 0$.  We use polar coordinates to obtain 
 \begin{align}
  \P(|W_t|\geq r) 
  &=(2\pi t)^{-\frac{d}{2}}\int_{B(r)^c}e^{-\frac{|x|^2}{2t}}\,\d x 
  =(2\pi t)^{-\frac{d}{2}}\int_r^\infty\int_{\S_u}e^{-\frac{|x|^2}{2t}}\,\d x\,\d u \nonumber \\
  &=(2\pi t)^{-\frac{d}{2}}|\S_1|\int_r^\infty u^{d-1}e^{-\frac{u^2}{2t}}\,\d u.  \label{eq:prob1}
 \end{align} 
 Choose $r_0\geq 1$ such that for all $u\geq r_0$ and $t\leq 1$, 
 \begin{equation}  \label{eq:prob2}
  e^{-\frac{u^2}{2t}}u^{d-1}  \leq e^{-\frac{2u}{t}}e^{-\frac{u}{t}}u^{d-1} \leq e^{-\frac{2u}{t}}.  
 \end{equation}
 Since $|\S_1|(2\pi t)^{-\frac{d}{2}}e^{-\frac{2u}{t}}\leq|\S_1|(2\pi t)^{-\frac{d}{2}}e^{-\frac{1}{t}}e^{-\frac{u}{t}}$
 for all $u\geq 1$ and $|\S_1|(2\pi t)^{-\frac{d}{2}}e^{-\frac{1}{t}}\to 0$ as $t\downarrow 0$,  there exists
 $t_0\in (0,1]$ such that 
 \begin{equation}  \label{eq:prob3}
  |\S_1|(2\pi t)^{-\frac{d}{2}}e^{-\frac{2u}{t}} \leq e^{-\frac{u}{t}} \quad\mbox{for all } t\in [0,t_0]. 
 \end{equation}
 Combing inequalities~\eqref{eq:prob1}-\eqref{eq:prob3} yields
 \[ \P(|W_t|\geq r)\leq \int_r^\infty e^{-\frac{u}{t}}\,\d u=te^{-\frac{r}{t}}
 	\quad\mbox{for all } t\in [0,t_0] \mbox{ and } r\geq r_0.  \qedhere \]
\end{proof}

In addition,  we need the following lemma which is similar to the estimates in~\cite[Example~5.3]{DKN19}
and in the proof of~\cite[Theorem~6.2]{BK20}.

\begin{lemma} \label{lem:est pq}
 Let $f\colon\R^d\to\R$ be a bounded,  measurable function.  Let $p,q\in (1,\infty)$ with 
 $\frac{1}{p}+\frac{1}{q}=1$. Then,  for every $\lambda, x\in\R^d$ and $t\geq 0$, 
 \[ \E[|f(x+W_t+\lambda t)|]\leq e^\frac{(q-1)|\lambda|^2t}{2}\E\left[f^p(x+W_t)\right]^\frac{1}{p}.  \]
\end{lemma}
\begin{proof}
 We use $W_t\sim\mathcal{N}(0,t\one)$ and the formula for the moment generating function of 
 the normal distribution to estimate 
 \begin{align*}
  &\E[|f(x+W_t+\lambda t)|] \\
  &=(2\pi t)^{-\frac{d}{2}}\int_{\R^d}|f(x+y+\lambda t)|\exp\left(-\frac{|y|^2}{2t}\right)\d y \\
  &=(2\pi t)^{-\frac{d}{2}}\int_{\R^d}|f(x+y)|\exp\left(-\frac{|y-\lambda t|^2}{2t}\right)\d y \\
  &=e^{-\frac{|\lambda|^2t}{2}}\int_{\R^d}|f(x+y)|\exp\left(-\langle\lambda,y\rangle\right)\n(0,t\one)(\d y) \\
  &\leq e^{-\frac{|\lambda|^2t}{2}}\left(\int_{\R^d}f^p(x+y)\,\n(0,t\one)(\d y)\right)^\frac{1}{p}
  	\left(\int_{\R^d}\exp\left(-q\langle\lambda,y\rangle\right)\n(0,t\one)(\d y)\right)^\frac{1}{q} \\
  &=e^{-\frac{|\lambda|^2t}{2}}\E\left[f^p(x+W_t)\right]^\frac{1}{p}
  	\E\left[\exp\left(q\langle\lambda,W_t\rangle\right)\right]^\frac{1}{q} \\
  &=e^{-\frac{|\lambda|^2t}{2}}\E\left[f^p(x+W_t)\right]^\frac{1}{p}e^{\frac{q|\lambda|^2t}{2}} 
  	=e^\frac{(q-1)|\lambda|^2t}{2}\E\left[f^p(x+W_t)\right]^\frac{1}{p}.  \qedhere
 \end{align*}
\end{proof}

Now, we are ready to extend the semigroup $(S(t))_{t\geq 0}$ on $C_0$ to a semigroup
$(\tilde{S}(t))_{t\geq 0}$ on $M^\Phi$.  In the sequel,  for notationally simplicity the extension 
will be again denoted by $(S(t))_{t\geq 0}$.  Furthermore, $A$ denotes the generator of
$(S(t))_{t\geq 0}$ w.r.t. the Luxembourg norm. In the proof of the following theorem we verify 
Assumption~\ref{ass:OSG} and show that $|S(t)f|\leq T(t)|f|$ holds for all $t\geq 0$ and 
$f\in C_0\cap M^\Phi$.  Hence,  we can apply Theorem~\ref{thm:OSG},  Theorem~\ref{thm:ext} 
and Theorem~\ref{thm:PDE}.

\begin{theorem} \label{thm:gexp ext}
 There exists a strongly continuous, convex, locally uniformly Lipschitz continuous semigroup 
 $(S(t))_{t\geq 0}$ on $M^\Phi$ with
 \[ \lim_{t\downarrow 0}\left\|\frac{S(t)f-f}{t}-\frac{1}{2}\Delta f-H(\nabla f)\right\|_\Phi=0
 	\quad\mbox{for all } f\in C^2_c.  \]
 In particular,  for every $f\in D(A)$,  the function $u(t):=S(t)f$, $t\geq 0$, solves the 
 abstract Cauchy problem~\eqref{eq:CP2} and satisfies 
 $u\in C^1([0,\infty);M^\Phi)\cap C([0,\infty);D(A))$. 
\end{theorem}
\begin{proof}
 First, we verify Assumption~\ref{ass:OSG}. Clearly, ~(i) holds. To prove~(ii), choose 
 $t_0>0$ and $r_0\geq 0$ such that we can apply Lemma~\ref{lem:prob}.  Let
 $r:=\max\{r_0,2bc\}$. Then,  
 \[ \P(|W_t\geq r)\Phi\left(\frac{c}{t}\right) \leq te^{-\frac{r}{t}}e^{\frac{2bc}{t}}=te^{\frac{2bc-r}{t}}
 	\to 0 \quad\mbox{as } t\downarrow 0.  \]
 	
 Second,  we show $|I(t)f|\leq T(t)|f|$ for all $t\geq 0$ and $f\in C_0\cap M^\Phi$.  Fix $t\geq 0$
 and $f\in C_0\cap M^\Phi$.  Let $\lambda\in\R^d$.  We distinguish between two cases.  In the 
 first case,  assume $|\lambda|\leq 2K$.  Then,  Lemma~\ref{lem:est pq} with $p=q=2$ implies
 \begin{align*}
 &\E[|f(x+W_t+\lambda t)|]-L(\lambda)t \leq\E[|f(x+W_t+\lambda t)|] \\
  &\leq e^{2K^2t}\E[f^2(x+W_t)]^\frac{1}{2} =\E[(e^{at}|f|)^2(x+W_t)|]^\frac{1}{2}.  
 \end{align*}
 Moreover, we apply Lemma~\ref{lem:ap} with $u(x):=x^2$ and $v(x):=\phi(x)$. Indeed,
 \[ \frac{u''(x)}{u'(x)}=\frac{1}{x}\leq\frac{bx+1}{x}=\frac{v''(x)}{v'(x)}.  \]
 Thus, we obtain 
 \[ \E[|f(x+W_t+\lambda t)|]-L(\lambda)t \leq\phi^{-1}\big(\E[\phi(e^{at}|f|(x+W_t))]\big)
 	=\big(T(t)|f|\big)(x).  \]
 In the second case,  assume $|\lambda|>2K$.  Choose $p:=8K+1$ and $q:=1+\frac{1}{8K}$. 
 It follows from Jensen's inequality,  inequality~\eqref{eq:L growth} and Lemma~\ref{lem:est pq} 
 that
 \begin{align*}
  \E[|f(x+W_t+\lambda t)|]-L(\lambda)t
  &\leq\log\big(\E[\exp(|f(x+W_t+\lambda t)|]\big)-\frac{|\lambda|^2t}{16K} \\
  &\leq\log\left(e^{\frac{(q-1)|\lambda|^2t}{2}}\E[\exp(p|f(x+W_t)|]^\frac{1}{p}\right)-\frac{|\lambda|^2t}{16K} \\
  &=\frac{1}{p}\log\big(\E[\exp(p|f(x+W_t)|]\big)+\left(\frac{q-1}{2}-\frac{1}{16K}\right)|\lambda|^2t \\
  &=\frac{1}{p}\log\big(\E[\exp(p|f(x+W_t)|]\big)
 \end{align*}
 Let $u(x):=e^{px}$ and $v(x):=\phi(x)$.  Since $b=p$, it holds
 \[ \frac{u''(x)}{u'(x)}=p\leq\frac{bx+1}{x}=\frac{v''(x)}{v'(x)}.  \]
 Hence, Lemma~\ref{lem:ap} implies $\E[|f(x+W_t+\lambda t)|]-L(\lambda)t\leq \big(T(t)|f|\big)(x)$. 
 For the lower bound,  we use that there exists $\lambda_0\in\R^d$ with $L(\lambda_0)=0$, to obtain
 \[ -\big(T(t)|f|\big)(x)\leq -\E[|f(x+W_t+\lambda_0 t)|]\leq \big(I(t)f\big)(x)\leq\big(T(t)|f|\big)(x),  \]
 and therefore $|I(t)f|\leq T(t)|f|$. 
 
 Third, fix $f\in C_0\cap M^\Phi$.  Let $s,t\geq 0$.  From the second part of this proof, monotonicity of
 $T(s)$ and and Theorem~\ref{thm:OSG}(iii), we obtain
 \[ |I(s)I(t)f|\leq T(s)|I(t)f|\leq T(s)T(t)|f|\leq T(s+t)|f|.  \]
 By induction, we conclude $|I(\pi_n^t)f|\leq T(t)|f|$ for all $t\in\T$ and $n\in\N$.  Thus, 
 Theorem~\ref{thm:gexp C_0}(i) implies $|S(t)f|\leq T(t)|f|$ for all $t\in\T$.  Let $t\geq 0$ be arbitrary and 
 $x\in\R^d$.  Choose a sequence $(t_n)_{n\in\N}\subset\T$ with $t_n\to t$.  It follows from 
 Theorem~\ref{thm:gexp C_0}(v) that $\lim_{n\to\infty}\|S(t)f-S(t_n)f\|_\infty=0$, and from the dominated
 convergence theorem that $(T(t)|f|)(x)=\lim_{n\to\infty}(T(t_n)|f|)(x)$.  This shows $|(S(t)f)(x)|\leq (T(t)|f|)(x)$. 
\end{proof}

\subsection{The symmetric Lipschitz set}
\label{sec:reg}

Lipschitz sets have been systematically studied in~\cite[Section~5]{BK20}, see also~\cite{DKN20}. 
We begin with the formal definition.

\begin{definition}
 The Lipschitz set $\L^S$ consists of all $f\in C_0\cap M^\Phi$ for which there exist a non-decreasing 
 function $\gamma(f,\cdot)\colon\R_+\to\R_+$ such that
 \[ \|S(t)f-f\|_\infty\leq\gamma(f,T)t \quad\mbox{for all } T\geq 0 \mbox{ and }  t\in [0,T].  \]
 The symmetric Lipschitz set is defined as $\L^S_{\sym}:=\{f\in\L^S\colon -f\in\L^S\}$.  
\end{definition}

Since $(S(t))_{t\geq 0}$ is nonlinear,  we have in general $\L^S_{\sym}\subsetneq\L^S$.  Similar to 
the domain of the generator,  the (symmetric) Lipschitz set is invariant under the semigroup.  Determining
$\L^S$ or $D(A)$ is very difficult, but fortunately it is possible for $\L^S_{\sym}$.  

Denote by $L^\infty:=L^\infty(\R^d;\R)$ the set of all bounded Borel measurable functions 
$f\colon\R^d\to\R$.  For $k\in\N$ and $p\in [1,\infty]$, let $W^{k,p}:=W^{k,p}(\R^d;\R)$ 
be the $L^p$-Sobolev space of order $k$ and $W^{k,p}_{\loc}:=W^{k,p}_{\loc}(\R^d;\R)$ 
the respective local Sobolev space. Furthermore, for $f\in W^{1,\infty}$, the 
weak Laplacian  exists in $L^\infty$ if there exists a function $g\in L^\infty$ such that
$\int_{\R^d}gh\,\d\lambda=-\int_{\R^d}\langle\nabla f,\nabla h\rangle\,\d\lambda$ 
for all $h\in C^\infty_c$. In this case, we define $\Delta f:=g$.

\begin{lemma} \label{lem:Lset}
 It holds $S(t)\colon\L^S_{\sym}\to\L^S_{\sym}$ for all $t\geq 0$,  where
 \begin{align*}
  \L^S_{\sym} &=\{f\in W^{1,\infty}\cap C_0\cap M^\Phi\colon\Delta f \mbox{ exists in } L^\infty\} \\
  &=\bigcap_{p\geq 1} \{f\in W^{2,p}_{\loc}\cap C_0\cap M^\Phi\colon \Delta f\in L^\infty\}. 
 \end{align*}
 For every $f\in\L^S_{\sym}$ and $T\geq 0$,  one can choose $\gamma(f,T):=\gamma(f)$ as a
 constant depending only on $\|\Delta f\|_\infty$ and $\|\nabla f\|_\infty$.  In addition,
 \begin{equation} \label{eq:Lset}
  \|S(s)f-S(t)f\|_\infty\leq\gamma(f)|s-t| \quad\mbox{for all } s,t\geq 0.  
 \end{equation}
\end{lemma}
\begin{proof}
 The characterization and invariance of $\L^S_{\sym}$ have been established 
 in~\cite[Theorem~6.5]{BK20}.  The choice of $\gamma(f,T)$ can be derived immediately from the 
 corresponding proof.  Inequality~\eqref{eq:Lset} follows from the proof of~\cite[Corollary~2.8]{BK20}.  
\end{proof}

On the one hand, we know from Theorem~\ref{thm:gexp ext} that the abstract Cauchy problem~\eqref{eq:CP2} 
has a solution which is represented by the semigroup.  On the other hand, we know from Theorem~\ref{lem:Lset} 
that for elements of the symmetric Lipschitz set the differential operator on the right-hand side of
equation~\eqref{eq:CP3} is well-defined.  The natural questions arises whether this differential
operator coincides with the generator.  The following theorem gives a positive answer to this
question.

\begin{theorem} \label{thm:reg}
 Let $f\in D(A)\cap\L^S_{sym}$ and define $u(t):=S(t)f$ for all $t\geq 0$.  Then,  $u$ solves the 
 Cauchy problem~\eqref{eq:CP3},  i.e.,  
 \[ Au(t)=\frac{1}{2}\Delta u(t)+H(\nabla u(t))  \quad\mbox{for all } t\geq 0,  \]
 where $u(t)\in \{g\in W^{1,\infty}\cap C_0\cap M^\Phi\colon\Delta g \mbox{ exists in } L^\infty\}$.
 Furthermore,  there exists a constant
 $C$ depending only on $\|\Delta f\|_\infty$,  $\|\nabla f\|_\infty$ and $H$ such that
 \begin{equation} \label{eq:u bound}
  \sup_{t\geq 0}\|\partial_t u(t)\|_\infty+\|\Delta u(t)\|_\infty+\|\nabla u(t)\|_\infty\leq C.
 \end{equation}
\end{theorem}
\begin{proof}
 We start by proving $\partial_t u(t)=\frac{1}{2}\Delta u(t)+H(\nabla u(t))$ for all $t\geq 0$.  
 First,  let $g\in\L^S_{\sym}$ with compact support.  We show that $g\in D(A)$ with
 $Ag=\frac{1}{2}\Delta g+H(\nabla g)$.  To do so,  let $\eta$ be a molifier,
 i.e. $\eta\in C_c^\infty$ with $\eta\geq 0$ and $\int_{\R^d}\eta\,\d\lambda=1$.  Define
 $\eta_n(x):=n^d\eta(nx)$ and $g_n:=g*\eta_n$ for all $n\in\N$ and $x\in\R^d$.  Since
 $g_n\in C^\infty_c$, Theorem~\ref{thm:gexp ext} implies $g_n\in D(A)$ with 
 $Ag_n=\frac{1}{2}\Delta g_n+H(\nabla g_n)$ for all $n\in\N$.  By Lemma~\ref{lem:Lset},
 the functions $\nabla g$ and $\Delta g$ are bounded and compactly supported. Hence,
 we can use the dominated convergence theorem~\cite[Theorem~14 in Chapter~3.4]{Rao}, to 
 conclude
 \[ \lim_{n\to\infty}\left\|\frac{1}{2}\Delta g+H(\nabla g)-\frac{1}{2}\Delta g_n-H(\nabla g_n)\right\|_\Phi=0.  \]
 Since $A$ is closed,  we obtain $g\in D(A)$ with $Ag=\frac{1}{2}\Delta g+H(\nabla g))$.
 Second,  let $\zeta$ be a cutting function,  i.e., $\zeta\in C^\infty_c$,  $0\leq\zeta\leq 1$ 
 and $\zeta=1$ on $B(1)$.  Define $\zeta_n(x)=\zeta(x/n)$ and $u_n(t):=\zeta_nu(t)$ 
 for all $n\in\N$, $x\in\R^d$ and $t\geq 0$.  Then,  it holds $u_n(t)\in\L^S_{\sym}$ with compact
 support for all $n\in\N$.  We use $\partial_t u_n(t)=\zeta_n\partial_t u(t)$ for all $n\in\N$, to the obtain
 \[ \partial_t u(t)=\lim_{n\to\infty}\partial_t u_n(t)
 	=\lim_{n\to\infty}\frac{1}{2}\Delta u_n(t)+H(\nabla u_n(t))=\frac{1}{2}\Delta u(t)+H(\nabla u(t)),  \]
 where the limits are understood pointwise.  

 It remains to show inequality~\eqref{eq:u bound}.  Fix $t\geq 0$.  We know from Theorem~\ref{thm:gexp ext} 
 that $u\in C^1([0,\infty);M^\Phi)$.  Furthermore, equation~\eqref{eq:Lset} yields
 \[ \left\|\frac{u(s)-u(t)}{s-t}\right\|_\infty\leq\gamma(f) \quad\mbox{for all } s\geq 0.  \]
 Since convergence w.r.t.  $\|\cdot\|_\Phi$ implies convergence pointwise $\lambda$-a.e.  along a 
 subsequence, the previous inequality is preserved for the limit $s\to t$,  i.e., 
 $\|\partial_t u(t)\|_\infty\leq\gamma(f)$.  From the proof of~\cite[Theorem~5.2]{BK20},  
 we conclude
 \begin{equation} \label{eq:Lset sym}
  \|S(s)(-u(t))+u(t)\|_\infty\leq 2\gamma(f)s \quad\mbox{for all } s\geq 0.  
 \end{equation}
 Indeed,  choose $S^+(t):=S(t)$,  $S^-(t):=-S(t)(-\,\cdot\,)$,  $\beta:=1$ and consider the last line 
 of the proof.  Thus, by using inequality~\eqref{eq:Lset} and inequality~\eqref{eq:Lset sym},  we can 
 proceed similar to the proof of~\cite[Theorem~6.5]{BK20} to estimate $\|\Delta u(t)\|_\infty$ and 
 $\|\nabla u(t)\|_\infty$.  
\end{proof}

In the proof of Theorem~\ref{thm:reg},  we have shown $\L^S_{\sym}\cap C_c\subset D(A)$. 
A complete characterization of $\L^S_{\sym}\cap D(A)$ is beyond the scope of this paper.  Note that 
$\L^S_{\sym}$ is the domain of the Laplacian in $L^\infty$ restricted to $C_0\cap M^\Phi$, 
see~\cite[Theorem~3.1.7]{Lunardi95}.  But except for the second equality in Lemma~\ref{lem:Lset},  
the results in this section are independent of established PDE theory.  Nonetheless, by relying on the 
$L^p$-theory for the Laplacian presented in~\cite{LSU1968}, we see that the solution $u$ from 
Theorem~\ref{thm:reg} has the additional regularity $u(t)\in\bigcap_{p\geq 1}W^{2,p}_{\loc}$ 
for all $t\geq 0$. Moreover, for every $r\geq 0$, there exists a constant $C_r\geq 0$ depending 
only on $C$ and $r$ such that
\[ \sup_{t\geq 0}\|u(t)\|_{W^{2,p}(B(r))}\leq C_r.  \]
This follows immediately from inequality~\eqref{eq:u bound} and~\cite[Theorem~3.1.6]{Lunardi95}.

\bibliographystyle{abbrv}
\bibliography{Orlicz}

\begin{thebibliography}{10}

\bibitem{applebaum2009}
D.~Applebaum.
\newblock {\em L\'{e}vy processes and stochastic calculus}, volume 116 of {\em
  Cambridge Studies in Advanced Mathematics}.
\newblock Cambridge University Press, Cambridge, second edition, 2009.

\bibitem{AT2015}
S.~N. Armstrong and H.~V. Tran.
\newblock Viscosity solutions of general viscous {H}amilton-{J}acobi equations.
\newblock {\em Math. Ann.}, 361(3-4):647--687, 2015.

\bibitem{Barbu2010}
V.~Barbu.
\newblock {\em Nonlinear differential equations of monotone types in {B}anach
  spaces}.
\newblock Springer Monographs in Mathematics. Springer, New York, 2010.

\bibitem{BSW2002}
M.~Ben-Artzi, P.~Souplet, and F.~B. Weissler.
\newblock The local theory for viscous {H}amilton-{J}acobi equations in
  {L}ebesgue spaces.
\newblock {\em J. Math. Pures Appl. (9)}, 81(4):343--378, 2002.

\bibitem{BC1991}
P.~B\'{e}nilan and M.~G. Crandall.
\newblock Completely accretive operators.
\newblock In {\em Semigroup theory and evolution equations ({D}elft, 1989)},
  volume 135 of {\em Lecture Notes in Pure and Appl. Math.}, pages 41--75.
  Dekker, New York, 1991.

\bibitem{BK20}
J.~Blessing and M.~Kupper.
\newblock Nonlinear semigroups built on generating families and their
  {L}ipschitz sets.
\newblock {\em Preprint}, 2020.

\bibitem{Brezis71}
H.~Br\'{e}zis.
\newblock Monotonicity methods in {H}ilbert spaces and some applications to
  nonlinear partial differential equations.
\newblock In {\em Contributions to nonlinear functional analysis ({P}roc.
  {S}ympos., {M}ath. {R}es. {C}enter, {U}niv. {W}isconsin, {M}adison, {W}is.,
  1971)}, pages 101--156. Academic Press, New York, 1971.

\bibitem{BH2006}
P.~Briand and Y.~Hu.
\newblock B{SDE} with quadratic growth and unbounded terminal value.
\newblock {\em Probab. Theory Related Fields}, 136(4):604--618, 2006.

\bibitem{CS2012}
P.~Cardaliaguet and L.~Silvestre.
\newblock H\"{o}lder continuity to {H}amilton-{J}acobi equations with
  superquadratic growth in the gradient and unbounded right-hand side.
\newblock {\em Comm. Partial Differential Equations}, 37(9):1668--1688, 2012.

\bibitem{chernoff1968}
P.~R. Chernoff.
\newblock Note on product formulas for operator semigroups.
\newblock {\em J. Functional Analysis}, 2:238--242, 1968.

\bibitem{chernoff1974}
P.~R. Chernoff.
\newblock {\em Product formulas, nonlinear semigroups, and addition of
  unbounded operators}, volume 140.
\newblock American Mathematical Soc., 1974.

\bibitem{CG2020a}
M.~Cirant and A.~Goffi.
\newblock Lipschitz regularity for viscous {H}amilton-{J}acobi equations with
  {$L^p$} terms.
\newblock {\em Ann. Inst. H. Poincar\'{e} Anal. Non Lin\'{e}aire},
  37(4):757--784, 2020.

\bibitem{CG2020}
M.~Cirant and A.~Goffi.
\newblock Maximal {$L^q$}-regularity for parabolic hamilton-jacobi equations
  and applications to mean field games, 2020.

\bibitem{CPX2017}
A.~Cosso, H.~Pham, and H.~Xing.
\newblock B{SDE}s with diffusion constraint and viscous {H}amilton-{J}acobi
  equations with unbounded data.
\newblock {\em Ann. Inst. Henri Poincar\'{e} Probab. Stat.}, 53(4):1528--1547,
  2017.

\bibitem{DGLS2006}
A.~Dall'Aglio, D.~Giachetti, C.~Leone, and S.~Segura~de Le\'{o}n.
\newblock Quasi-linear parabolic equations with degenerate coercivity having a
  quadratic gradient term.
\newblock {\em Ann. Inst. H. Poincar\'{e} Anal. Non Lin\'{e}aire},
  23(1):97--126, 2006.

\bibitem{Davini19}
A.~Davini.
\newblock Existence and uniqueness of solutions to parabolic equations with
  superlinear {H}amiltonians.
\newblock {\em Commun. Contemp. Math.}, 21(1):1750098, 25, 2019.

\bibitem{DHX2011}
F.~Delbaen, Y.~Hu, and X.~Bao.
\newblock Backward {SDE}s with superquadratic growth.
\newblock {\em Probab. Theory Related Fields}, 150(1-2):145--192, 2011.

\bibitem{DKN20}
R.~Denk, M.~Kupper, and M.~Nendel.
\newblock Convex monotone semigroups on lattices of continuous functions, 2020.

\bibitem{DKN19}
R.~Denk, M.~Kupper, and M.~Nendel.
\newblock Convex semigroups on {$L^p$}-like spaces.
\newblock {\em Forthcoming in Journal of Evolution Equations}, 2020.

\bibitem{DKN20a}
R.~Denk, M.~Kupper, and M.~Nendel.
\newblock A semigroup approach to nonlinear {L}\'evy processes.
\newblock {\em Stochastic Process. Appl.}, 130:1616--1642, 2020.

\bibitem{Evans78}
L.~C. Evans.
\newblock Application of nonlinear semigroup theory to certain partial
  differential equations.
\newblock In {\em Nonlinear evolution equations ({P}roc. {S}ympos., {U}niv.
  {W}isconsin, {M}adison, {W}is., 1977)}, volume~40 of {\em Publ. Math. Res.
  Center Univ. Wisconsin}, pages 163--188. Academic Press, New York-London,
  1978.

\bibitem{FPR2001}
V.~Ferone, M.~R. Posteraro, and J.~M. Rakotoson.
\newblock Nonlinear parabolic problems with critical growth and unbounded data.
\newblock {\em Indiana Univ. Math. J.}, 50(3):1201--1215, 2001.

\bibitem{Foellmer}
H.~F\"{o}llmer and A.~Schied.
\newblock {\em Stochastic finance}.
\newblock De Gruyter Graduate. De Gruyter, Berlin, 2016.
\newblock An introduction in discrete time, Fourth revised and extended edition
  of [ MR1925197].

\bibitem{Gilding2005}
B.~H. Gilding.
\newblock The {C}auchy problem for {$u_t=\Delta u+|\nabla u|^q$}, large-time
  behaviour.
\newblock {\em J. Math. Pures Appl. (9)}, 84(6):753--785, 2005.

\bibitem{GGK2003}
B.~H. Gilding, M.~Guedda, and R.~Kersner.
\newblock The {C}auchy problem for {$u_t=\Delta u+|\nabla u|^q$}.
\newblock {\em J. Math. Anal. Appl.}, 284(2):733--755, 2003.

\bibitem{Kato67}
T.~Kato.
\newblock Nonlinear semigroups and evolution equations.
\newblock {\em J. Math. Soc. Japan}, 19:508--520, 1967.

\bibitem{Kobylanski2000}
M.~Kobylanski.
\newblock Backward stochastic differential equations and partial differential
  equations with quadratic growth.
\newblock {\em Ann. Probab.}, 28(2):558--602, 2000.

\bibitem{LSU1968}
O.~A. Lady\v{z}enskaja, V.~A. Solonnikov, and N.~N. Ural'ceva.
\newblock {\em Linear and quasilinear equations of parabolic type}.
\newblock Translated from the Russian by S. Smith. Translations of Mathematical
  Monographs, Vol. 23. American Mathematical Society, Providence, R.I., 1968.

\bibitem{Lunardi95}
A.~Lunardi.
\newblock {\em Analytic semigroups and optimal regularity in parabolic
  problems}, volume~16 of {\em Progress in Nonlinear Differential Equations and
  their Applications}.
\newblock Birkh\"{a}user Verlag, Basel, 1995.

\bibitem{meyer1991}
P.~Meyer-Nieberg.
\newblock {\em Banach lattices}.
\newblock Universitext. Springer-Verlag, Berlin, 1991.

\bibitem{NR19}
M.~Nendel and M.~R\"ockner.
\newblock Upper envelopes of families of {F}eller semigroups and viscosity
  solutions to a class of nonlinear {C}auchy problems.
\newblock {\em Preprint arXiv:1906.04430}, 2019.

\bibitem{Nisio76}
M.~Nisio.
\newblock On a non-linear semi-group attached to stochastic optimal control.
\newblock {\em Publ. Res. Inst. Math. Sci.}, 12(2):513--537, 1976/77.

\bibitem{Rao}
M.~M. Rao and Z.~D. Ren.
\newblock {\em Theory of {O}rlicz spaces}, volume 146 of {\em Monographs and
  Textbooks in Pure and Applied Mathematics}.
\newblock Marcel Dekker, Inc., New York, 1991.

\bibitem{sato1999}
K.~Sato.
\newblock {\em L\'{e}vy processes and infinitely divisible distributions},
  volume~68 of {\em Cambridge Studies in Advanced Mathematics}.
\newblock Cambridge University Press, Cambridge, 1999.
\newblock Translated from the 1990 Japanese original, Revised by the author.

\bibitem{SZ2006}
P.~Souplet and Q.~S. Zhang.
\newblock Global solutions of inhomogeneous {H}amilton-{J}acobi equations.
\newblock {\em J. Anal. Math.}, 99:355--396, 2006.

\bibitem{trotter1958}
H.~F. Trotter.
\newblock Approximation of semi-groups of operators.
\newblock {\em Pacific J. Math.}, 8:887--919, 1958.

\bibitem{trotter1959}
H.~F. Trotter.
\newblock On the product of semi-groups of operators.
\newblock {\em Proc. Amer. Math. Soc.}, 10:545--551, 1959.

\end{thebibliography}

\end{document}